\documentclass[12pt]{amsart}
\usepackage{amssymb,amsfonts,amsmath,amsopn,amstext,amscd,color,
graphicx,latexsym,mathabx,mathrsfs,skak,stackrel,todonotes,xy,url,
verbatim,tikz-cd,enumitem}

\usepackage{hyperref}

\input xy
\xyoption{all}

\theoremstyle{remark}

\newtheorem{para}{\bf}[subsection]
\newtheorem{example}[para]{\bf Example}

\newtheorem{rem}[para]{\bf Remark}
\theoremstyle{definition}

\newtheorem{dfn}[para]{Definition}

\theoremstyle{plain}

\newtheorem{thm}[para]{Theorem}

\newtheorem{prop}[para]{Proposition}

\makeatletter
\newcommand*\bdot{{\mathpalette\bdot@{.9}}}
\newcommand*\bdot@[2]{\mathbin{\vcenter{\hbox{\scalebox{#2}{$\m@th#1\bullet$}}}}}
\makeatother

\usepackage{pdfpages}

\begin{document}

\title{Rational Exceptional Belyi Coverings}
\author{Cemile Kurkoglu}
\address{Denison University, OH, U.S.A.}
\email{kurkogluc@denison.edu}

\begin{abstract} Exceptional Belyi covering is a connected Belyi covering uniquely determined by its ramification scheme or the respective dessin d’enfant. We focus on rational exceptional Belyi coverings of compact Riemann surfaces of genus 0. Well known examples are cyclic, dihedral, and Chebyshev coverings. Using Maple, we identified all rational exceptional Belyi coverings up to degree 15. Their Belyi functions were calculated for degrees up to 6 along with some for degree 7. We also found new infinite series. 
\end{abstract}

\maketitle

\tableofcontents

\section{Introduction}

A \textit{Belyi covering} is a ramified covering of the Riemann sphere with ramification points belonging to the set $\{0,\,1,\,\infty\}\,.$ The notion \textit{dessin d’enfant} which is a French synonym for ``child-drawing” to study Belyi coverings. The following example from \cite{SinSyd} is a dessin d’enfant and gives us an idea about why this name is indeed chosen:
\begin{figure}[htp]
\includegraphics[width=30mm]{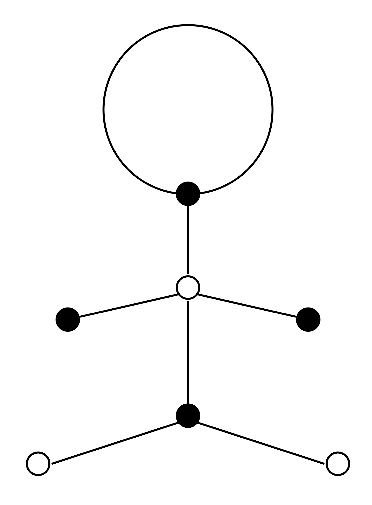}
\end{figure}

\subsection{Overview}

In Chapter 2, we start with some preliminaries about ramified coverings and then give the definition of Belyi coverings. We give examples of Shabat polynomials \cite{Sha}. We also state the famous Belyi’s Theorem saying that there exists a Belyi function on a compact surface that can be defined over a number field as an algebraic curve \cite{Belyi}. Later, we introduce dessin d'enfants as some connected bipartite maps. They have a topological structure too as mentioned in \cite{GirGon} and \cite{SinSyd}. The theory of dessins was first introduced by Grothendieck \cite{Groth}. Dessins and Belyi coverings are in one-to-one correspondence \cite{Zvo}. We also mention the action of the absolute Galois group on dessins as in \cite{LanZvo}, \cite{GirGon} annd \cite{CouGran}. We conclude this chapter by giving examples from \cite{ShaVoe}, \cite{BetZvo} and \cite{Zvo}.\\

Next we deal with how to count these coverings in Chapter 3. We define the Eisenstein number of a covering  \cite{LieSha}. Dessin of a polynomial covering is a bicolored plane tree. They were first studied by G. Shabat \cite{Sha}. Counting these trees is a combinatorial problem solved by Tutte \cite{Tut}. We then give a formula for counting coverings with given ramification schemes. This formula includes both Eisentein number and irreducible characters of symmetric groups \cite{KlyKur}. Preliminaries from representation theory can be found in \cite{Kurk}. \\

In Chapter 4; we focus on Belyi coverings that are uniquely determined by either ramification schemes or by respective dessin d'enfants. We call them \textit{exceptional Belyi coverings}. We characterize them with their ramification indices in fibers 0,1, and $\infty$. We give the well-known examples that are cyclic, dihedral and Chebyshev coverings. Cyclic coverings, dihedral coverings and the coverings of regular polyhedra (\cite{CouGran}, \cite{MagZvo}) are known as Klein's coverings. We found a number of rational exceptional coverings which include Klein's coverings and Chebyshev polynomials. We also found new infinite series. We give a classification of exceptional polynomial coverings by Adrianov (\cite{A09a}). Then we state the formula to count the exceptional Belyi coverings. Lastly, we talk about the \textit{field of definition}, which is the smallest field in which both the Riemann surface, or equivalently the corresponding algebraic curve, and the coefficients of the covering are defined. We state the field of definition of a rational exceptional Belyi covering with the help of theorems in \cite{Serre2}  and \cite{Serre3}. We conclude with our Maple algorithm finding all coverings with a given genus and degree. A table of all rational exceptional Belyi coverings up to degree 6 with ramification schemes, Belyi functions and dessins d'enfants and some of the ones that are of degree 7 is given. 

\section{Belyi coverings}

\subsection{Preliminaries}
Let $X$ and $Y$ be two manifolds.
\begin{dfn}
    A continuous map $p : Y \rightarrow X$ is said to be a \textit{covering} of $X$ if for every $x \in X\,,$ there exists an open neighborhood $U$ of $x$ such that
    $p^{-1}(U)$ is the disjoint union of open sets $\{V_i\}_{i\in\mathbb{R}}$, where $p|_{Vi}$ is a homeomorphism.
\end{dfn}
\begin{rem}
    Ramified coverings from one Riemann surface to another take the form of holomorphic maps that locally look like $x^k\,.$
\end{rem}
\begin{dfn}
    A map $p : Y \rightarrow X$ is said to be a \textit{ramified covering} of $X$ if there exist finite subsets $S\subset Y\,,$ $T \subset X$ such that $p|_{Y\setminus S} :Y \setminus S \rightarrow X\setminus T$ is a covering of $X \setminus T\,.$
\end{dfn}
Let $\mathbb{P}^{1}$ denote the Riemann sphere and let $S$ be a compact Riemann surface.
\begin{prop}
 A non-constant meromorphic function $f: S \rightarrow \mathbb{P}^{1}$ is a \textit{ramified covering} of $\mathbb{P}^{1}$.   
\end{prop}
Now let $f: S \rightarrow \mathbb{P}^{1}$ be a covering of the Riemann sphere ramified at $k$ points $s_{1}, \ldots, s_{k}$ with degree $n\,.$  The fundamental group is $\left\langle\gamma_{1}, \gamma_{2}, \ldots, \gamma_{k} \mid \gamma_{1} \gamma_{2} \ldots \gamma_{k}=i d\right\rangle$ and the monodromy group, as a subgroup of $S_n$ is $\left\langle g_{1}, g_{2}, \ldots, g_{k}\right| g_{1} g_{2} \ldots g_{k}=$ $i d\rangle\,.$ For each fiber $f^{-1}\left(s_{i}\right)$, let the cycle structure of $g_{i}$ be denoted by $\lambda_{i}\,.$ $\lambda_{i}$ 's are the \textit{ramification indices} and the expression $\left[\lambda_{1}\right]\left[\lambda_{2}\right] \ldots\left[\lambda_{k}\right]$ is called the \textit{ramification scheme} of the covering.

\subsection{Belyi coverings and Belyi theorem}

\begin{dfn} Let $S_g$ be a compact Riemann surface with genus $g$\,. A {\it Belyi covering} $\beta: S_g \to \mathbb{P}^1$ is a ramified covering of $\mathbb{P}^1$ with ramification points in $\{0,\,1,\,\infty\}\,.$
\end{dfn}
\vskip5pt
\begin{dfn}
    Let  $\beta: S_{g} \rightarrow \mathbb{P}^{1}$ be a Belyi covering.
\begin{itemize}
    \item If $g=0, \beta$ is called a \textit{rational} Belyi covering.
    \item If $g=1, \beta$ is called an \textit{elliptic} Belyi covering.
    \item If $g>1, \beta$ is called a \textit{hyperbolic} Belyi covering.
\end{itemize}
\end{dfn}

\vskip5pt

\begin{dfn}
    We will call $\beta$ a {\it Belyi function} when the genus $g = 0$ and in other cases, we will not only express $\beta$, instead we will write $(S_g,\,\beta)$ and call it a {\it Belyi pair}.
\end{dfn}

\begin{example} The following are some examples of Belyi functions:
\vspace{0.1cm}
\begin{itemize}
\item  Let $S=\mathbb{P}^{1}$.

$$
\beta: z \mapsto z^{n}
$$ \\
\item If $S=\mathbb{P}^{1}$, then consider the \textit{Belyi polynomial}

$$
\beta_{m, n}=z \mapsto \frac{1}{\mu} z^{m}(1-z)^{n}
$$
where $\mu=\frac{m^{m} n^{n}}{(m+n)^{m+n}}\,.$ \\
\item The $n$-th Chebyshev polynomial 
$T_{n}(z)=\cos n(\arccos z)$

\end{itemize}
   
\end{example}

\begin{rem} The polynomials above are ramified at $\infty\,.$ They are \textit{Shabat polynomials} as they are in general polynomials with at most two critical values.
\end{rem}

\vskip5pt

\begin{thm}
  {\bf (Belyi)} Let $S$ be a compact Riemann surface. The following statements are equivalent:

(a) $S$ is defined over $\overline{\mathbb{Q}}$.

(b) $S$ admits a meromorphic function $f: S \rightarrow \mathbb{P}^{1}$ with at most three ramification points.  
\end{thm} 

\begin{proof}
The proof can be found in \cite{Kurk} which is based on the ideas in \cite{GirGon} and \cite{Kock}.
\end{proof}

\vskip5pt

\subsection{Dessins d’enfants and Belyi coverings}
\begin{dfn}
 A \textit{dessin d'enfant}, or simply a \textit{dessin}, is a pair $(X, \mathcal{D})$ where $X$ is an oriented compact topological surface, and $\mathcal{D} \subset X$ is a finite graph such that:

(i) $\mathcal{D}$ is connected.

(ii) $\mathcal{D}$ is bicoloured.

(iii) $X \backslash \mathcal{D}$ is the union of finitely many topological discs, which we call faces of $\mathcal{D}$.   
\end{dfn}

 A dessin is a connected bipartite map with a topological structure. When the underlying surface is clear, we simply express a dessin as $\mathcal{D}$.
The genus of $(X, \mathcal{D})$ is simply the genus of the topological surface $X$.

\begin{dfn} We consider two dessins $\left(X_{1}, \mathcal{D}_{1}\right)$ and $\left(X_{2}, \mathcal{D}_{2}\right)$ equivalent when there exists an orientation-preserving homeomorphism from $X_{1}$ to $X_{2}$ whose restriction to $\mathcal{D}_{1}$ induces an isomorphism between the coloured graphs $\mathcal{D}_{1}$ and $\mathcal{D}_{2}$.
\end{dfn}

Now suppose that the edges of the dessin are numbered from the set $\Omega=$ $\{1,2,3, \ldots\}$. Each edge joins a black vertex to a white vertex, and incident with every black vertex, we have some of these edges. Using the anticlockwise orientation of the surface, we get a cyclic permutation of these edges. Thus if we have $b$ black vertices, we have a permutation $\sigma_{0}$ that is a product of $b$ disjoint cycles. Similarly, if we have $w$ white vertices then we get a permutation $\sigma_{1}$ consisting of $w$ disjoint cycles. The permutation $\sigma_{2}:=\left(\sigma_{0} \sigma_{1}\right)^{-1}$ describes the edges going around a face, each cycle of length $u$ corresponds to a $2 u$-gonal face.

\begin{example}
    Let $\sigma_{0}=(1248)(365)(7)$ and $\sigma_{1}=(1)(23)(4567)(8)$. So $\sigma_{2}=$ $(18473)(25)(6)$ and the corresponding diagram is the following:
\begin{figure}[htp]
\includegraphics[width=40mm]{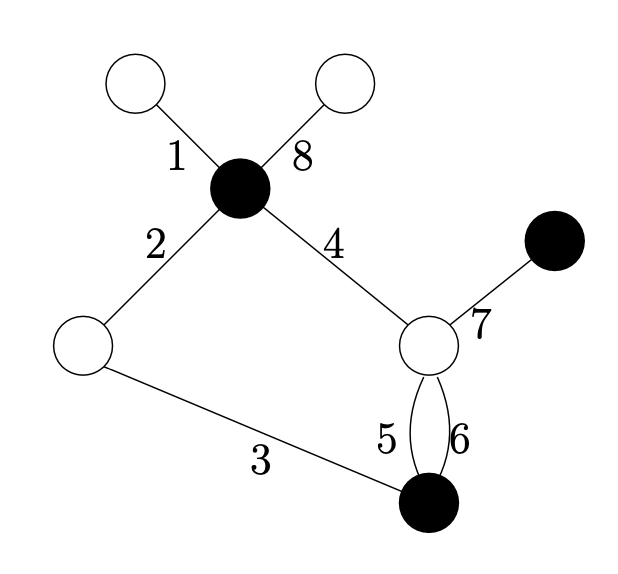}
\end{figure}
\end{example}

\begin{dfn}
    $\left\langle\sigma_{0}, \sigma_{1}\right\rangle$ is called the \textit{permutation representation pair} of the dessin.
\end{dfn}

\begin{thm}
There is a 1-1 correspondence between the equivalence classes of dessins and the equivalence classes of Belyi pairs.
\end{thm}

\begin{proof}
    It is proved using the concepts of graph theory and the Riemann Existence Theorem as in \cite{Kurk}.
\end{proof}

\begin{prop}
    The permutation representation pair of a dessin d'enfant and the monodromy of the corresponding Belyi pair are determined by each other.
\end{prop}

\begin{rem}Let the monodromy group of a Belyi covering be

$$
\left\langle g_{0}, g_{1}, g_{\infty} \mid g_{0}, g_{1}, g_{\infty}=i d\right\rangle\,.
$$

So the permutation representation pair for the corresponding dessin will be $\left\langle g_{0}, g_{1}\right\rangle$ and $g_{\infty}=\left(g_{0} g_{1}\right)^{-1}$. If the cycle structure of $g_{i}$ are $\lambda_{i}$, then the Belyi covering is determined by its \textit{ramification scheme} $\left[\lambda_{\infty}\right]\left[\lambda_{0}\right]\left[\lambda_{1}\right]$.
\end{rem}

\begin{thm}
    There is a one-to-one correspondence between the followings:
\begin{itemize}
    \item $\operatorname{dessins}\left(\mathcal{D}, S_{g}\right)$
    \item Belyi coverings of $\beta: S_{g} \rightarrow \mathbb{P}^{1}$ with degree $n$
    \item the solutions of the monodromy group relation $g_{0} g_{1} g_{\infty}=i d$, where $g_{i} \in S_{n}$
\end{itemize}
\end{thm}

\vskip5pt

\subsection{The action of the absolute Galois group 
\texorpdfstring{$\text{Gal}(\bar{\mathbb{Q}}/\mathbb{Q})$}{}}

\begin{dfn}
    The universal Galois group, or the absolute Galois group is the group of automorphisms of algebraic numbers $\overline{\mathbb{Q}}$ and denoted by $\Gamma=$ $\operatorname{Gal}(\overline{\mathbb{Q}} / \mathbb{Q}) \cdot \mathbb{Q}$ is the fixed field of $\Gamma$.
\end{dfn}

Let $k \subset \overline{\mathbb{Q}}$ be a number field. Every automorphism $k$ may be extended to an automorphism of $\overline{\mathbb{Q}}$.  Subgroups of $\Gamma$ of finite index are in one-to-one correspondence with finite extensions of $\mathbb{Q}$ inside $\overline{\mathbb{Q}}$.

\begin{rem}
    All orbits of the action of $\Gamma$ on dessins are finite.
\end{rem}

\begin{thm}
    Let $\mathcal{D}$ be a dessin. The following properties of $\mathcal{D}$ remain invariant under the action of $\Gamma$:
\begin{enumerate}
    \item the number of edges
    \item the number of white vertices, black vertices and faces
    \item the degree of the white vertices, black vertices and faces
    \item the genus
    \item the monodromy group
    \item the automorphism group
\end{enumerate}
\end{thm}

We conclude this section with the following theorem describing another facet of the action of $\Gamma$ :

\begin{thm}
    The restriction of the action of $\Gamma$ to dessins of genus $g$ is faithful for every $g$.
\end{thm}

\vskip5pt

\subsection{Examples}

\begin{example}
   In the figures below the Belyi functions and ramification schemes are given with the corresponding dessins for star-trees and for Chebyshev polynomials. 
\begin{figure}[htp]
\includegraphics[width=120mm]{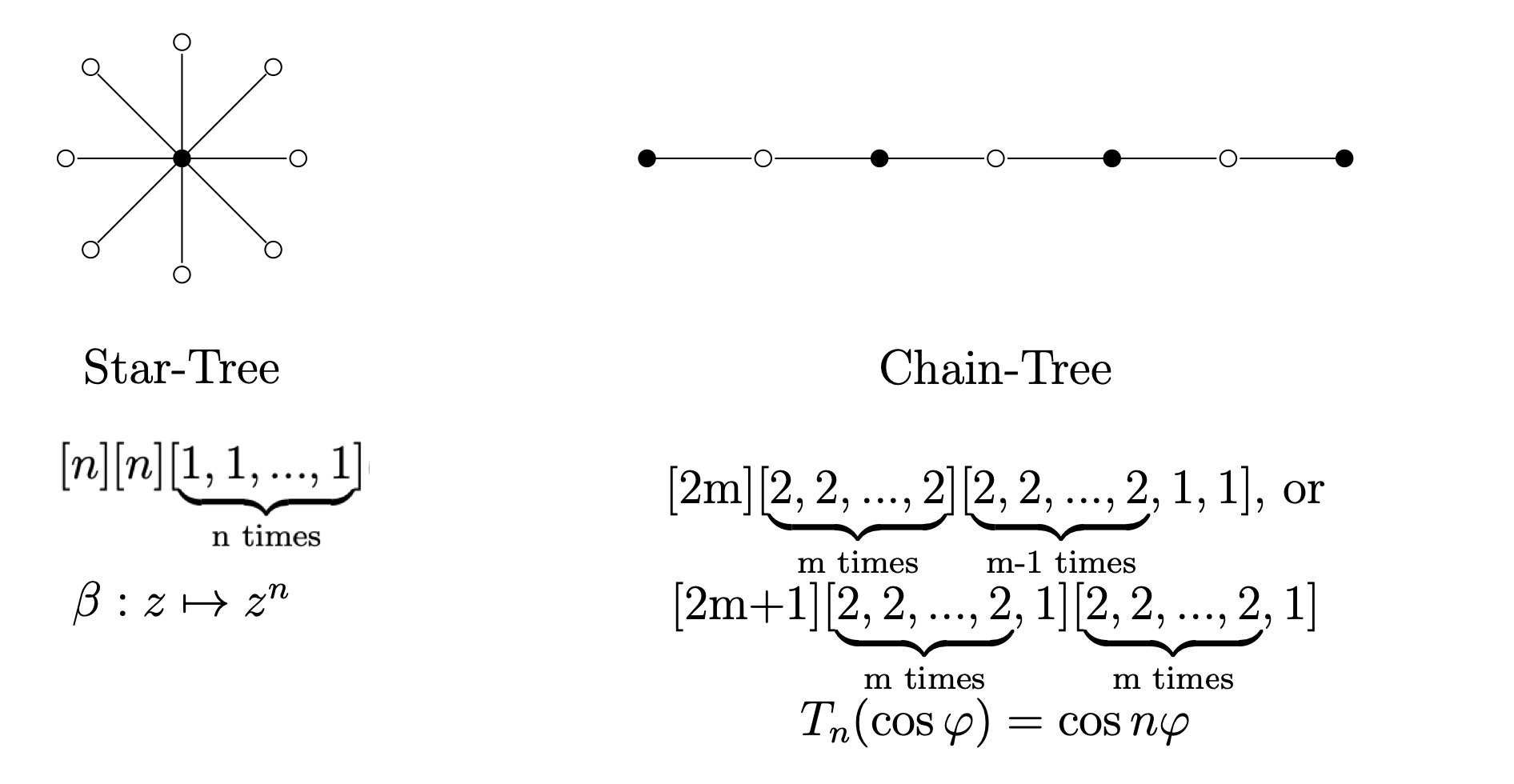}
\end{figure}
\end{example}

\begin{example}
There are ``conjugate" trees with a common ramification scheme: The following trees have the ramification scheme $[7][3,2,2][2,2,1,1,1]$ with the corresponding polynomials $P(x)=x^{3}\left(x^{2}-2 x \pm a\right)^{2}$, where $a=\frac{1}{7}(34 \pm 6 \sqrt{21})$ and they are both defined over the field $\mathbb{Q}(\sqrt{21})$.
\begin{figure}[htp]
\includegraphics[width=120mm]{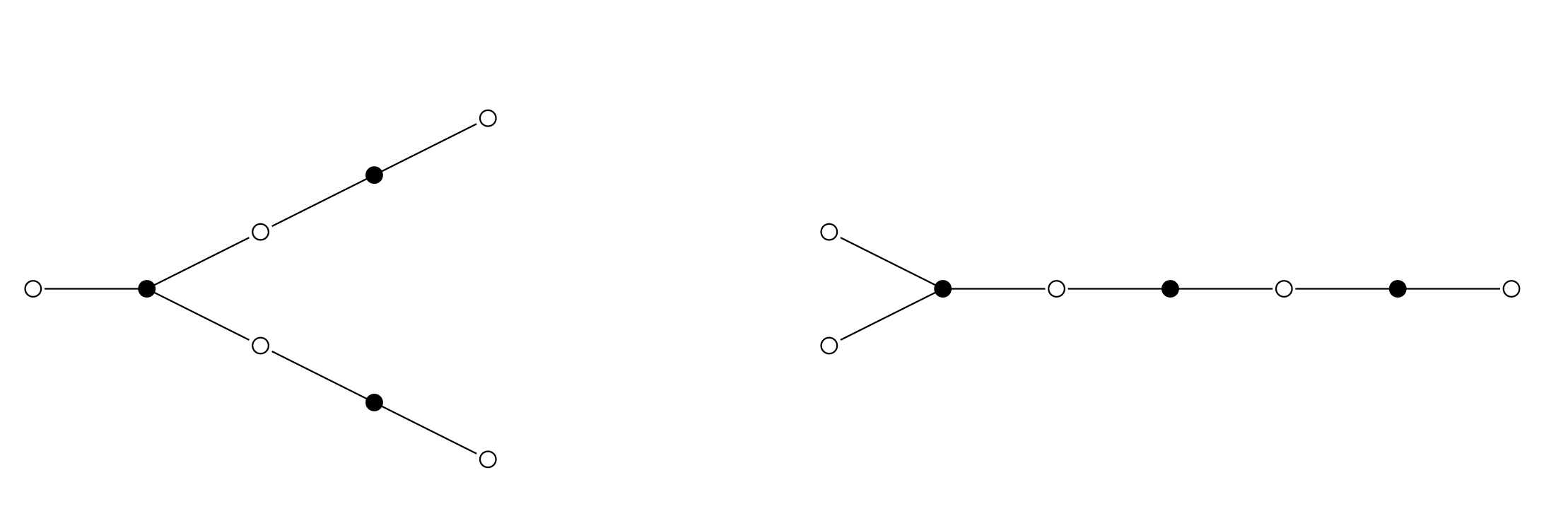}
\end{figure}
\end{example}

\begin{example}
    The following three dessins are defined over cubic fields, permutable by $\Gamma=\operatorname{Gal}(\overline{\mathbb{Q}} / \mathbb{Q})$. They lie in the decomposition of the polynomial

$$
25 x^{3}-12 x^{2}-24 x-16=25(x-a)\left(x-a_{+}\right)\left(x-a_{-}\right) \text {. }
$$

We agree that $a \in \mathbb{R}, \operatorname{Im}\left(a_{+}\right)>0, \operatorname{Im}\left(a_{-}\right)<0$.
\begin{figure}[htp]
\includegraphics[width=160mm]{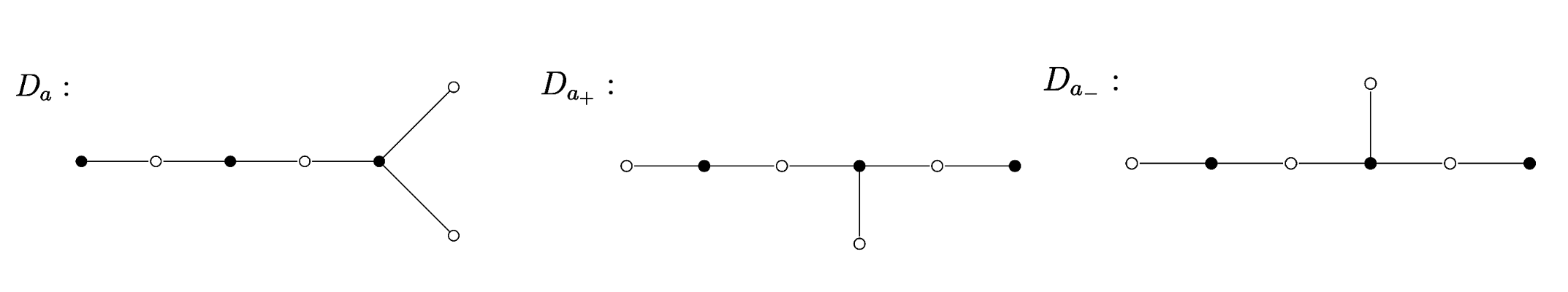}
\end{figure}
\hspace{0.1cm}

\end{example}

\begin{example}Now we will give an example of a rational Belyi covering which is not a tree.
  Let the dessin be as follows

\begin{figure}[htp]
\includegraphics[width=90mm]{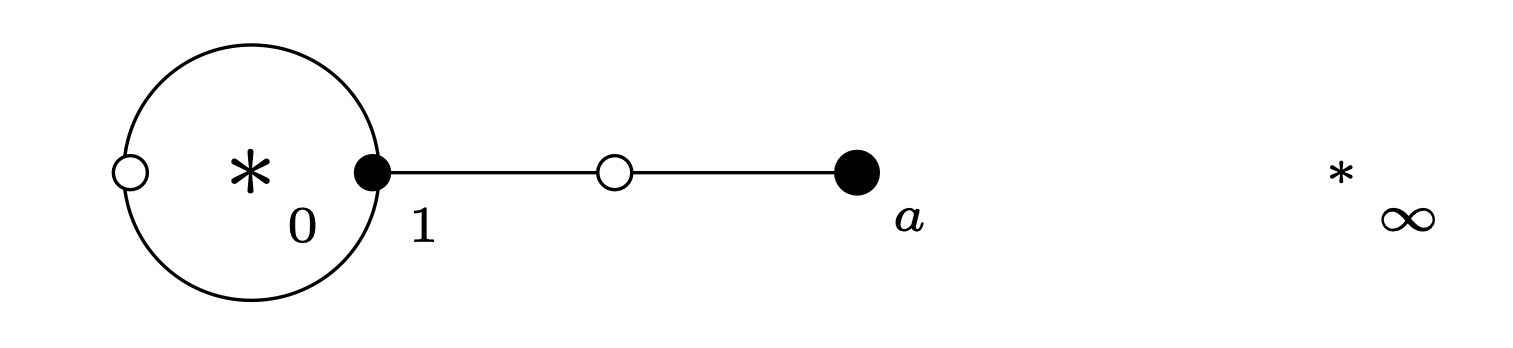}
\end{figure}

  One of the poles is at the center of one of the faces and the other pole is at $\infty$. We denote the roots of the corresponding function with 1 and a. Now the Belyi covering will be of the form

$$
\beta(x)=K \frac{(x-1)^{3}(x-a)}{x}
$$

with $K=-\frac{1}{64}$ and $a=9$. The roots of $\beta(z)-1$ corresponding to white vertices are $3 \pm 2 \sqrt{3}$.
\end{example}

\vskip5pt

\section{Counting coverings with a given ramification scheme}

\subsection{Tutte formula for counting polynomial coverings}

\begin{dfn}
     Let $\beta$ be a Belyi covering of $\mathbb{P}^{1}$. The centralizer in $S_{n}$ of the monodromy group of $\beta$ is called the \textit{automorphism group} of $\beta$ and denoted by Aut $\beta$.
\end{dfn}

\begin{dfn}
    Let $S$ be a compact Riemann surface. $\sum_{\beta: S \rightarrow \mathbb{P}^{1}} \frac{1}{\text { Aut } \beta}$ is called the \textit{Eisenstein number} of coverings of $\mathbb{P}^{1}$.
\end{dfn}

Tutte found the Eisenstein number of planar trees with $n$ edges and given degrees $d_{i}^{\bullet}$ and $d_{j}^{\circ}$ of black and white vertices. Clearly, $\sum_{i} d_{i}^{\bullet}=\sum_{j} d_{j}^{\circ}=n$, the number of edges of $T$, or what is the same, the degree of the respective covering. The degrees $d^{\bullet}=\left(d_{1}^{\bullet}, d_{2}^{\bullet}, \ldots\right)$ and $d^{\circ}=\left(d_{1}^{\circ}, d_{2}^{\circ}, \ldots\right)$ give two partitions of $n$. In practice it is more convenient to deal with partitions $q^{\bullet}=\left(q_{1}^{\bullet}, q_{2}^{\bullet}, \ldots\right)$ and $q^{\circ}=\left(q_{1}^{\circ}, q_{2}^{\circ}, \ldots\right)$ where $q_{i}^{\bullet}$ and $q_{i}^{\circ}$ is the number of black and white vertices of degree $i$. Observe that

\begin{equation*}
n=\sum_{i} i q_{i}^{\bullet}=\sum_{j} j q_{j}^{\circ}=\sum_{i} q_{i}^{\bullet}+\sum_{j} q_{j}^{\circ}-1 
\end{equation*}

We'll often use the last two sums and introduce for them special notations

\begin{equation*}
\sigma^{\bullet}=\sum_{i} q_{i}^{\bullet}, \quad \sigma^{\circ}=\sum_{j} q_{j}^{\circ} . 
\end{equation*}

In this notations the (slightly modified) Tutte result may be stated as follows:

\begin{thm}{\bf(Tutte formula)}
    $$
\sum_{T} \frac{1}{|\operatorname{Aut} T|}=\frac{1}{\sigma^{\bullet} \sigma^{\circ}}\left(\begin{array}{l}
\sigma^{\bullet}  \\
q^{\bullet}
\end{array}\right)\left(\begin{array}{l}
\sigma^{\circ} \\
q^{\circ}
\end{array}\right)
$$

where the sum is extended over all planar trees $T$ with given degrees $d^{\bullet}, d^{\circ}$ of black and white vertices. Parentheses in right hand side $\left(\begin{array}{c}\sigma^{\bullet} \\ q^{\bullet}\end{array}\right)$ stand for multinomial coefficient $\left(\begin{array}{c}q_{1}^{\bullet}+q_{2}^{\bullet}+\ldots+q_{k}^{\bullet} \\ q_{1}^{\bullet}, q_{2}^{\bullet}, \ldots, q_{k}^{\bullet}\end{array}\right)$.

\end{thm}

\subsection{Burnside Theorem}

Topologically, coverings $\pi: X \rightarrow Y$ of degree $n$ unramified outside $k$ points $y_{i} \in$ $Y$ are classified by conjugacy classes of homomorphisms $\pi_{1} \rightarrow S_{n}$ of the fundamental group $\pi_{1}=\pi_{1}\left(Y \backslash\left\{y_{1}, y_{2}, \ldots, y_{k}\right\}\right)$, which is known to be defined by the unique relation

$$
c_{1} c_{2} \ldots c_{k}\left[f_{1}, h_{1}\right]\left[f_{2}, h_{2}\right] \ldots\left[f_{g}, h_{g}\right]=1, \quad f_{i}, h_{i} \in S_{n}, c_{i} \in C_{i}
$$

where $g$ is the genus of $Y$ and the brackets denote the commutator $[f, h]=f h f^{-1} h^{-1}$. Thus the coverings of Riemann sphere $\pi: X \rightarrow \mathbb{P}^{1}$ of given degree $n$ and ramification indices are parametrized by solutions of the equation

\begin{equation*}
c_{1} c_{2} \ldots c_{k}=1, \quad c_{i} \in C_{i} 
\end{equation*}

up to conjugacy, where cycle lengths of the conjugacy class $C_{i} \subset S_{n}$ are equal to ramification indices of points in fibers $\pi^{-1}\left(y_{i}\right)$.

The following theorem gives the number of solutions of the equation above for an arbitrary group $G$ in terms of irreducible characters:

\begin{thm} {\bf (Burnside)}
\begin{equation*}
\#\left\{c_{1} c_{2} \ldots c_{k}=1 \mid c_{i} \in C_{i}\right\}=\frac{\left|C_{1}\right|\left|C_{2}\right| \ldots\left|C_{k}\right|}{|G|} \sum_{\chi} \frac{\chi\left(c_{1}\right) \chi\left(c_{2}\right) \ldots \chi\left(c_{k}\right)}{\left(\chi(1)^{k-2}\right)} \tag{1}\label{Burnside}
\end{equation*}
\end{thm}

\subsection{Eisenstein number of coverings and characters of \texorpdfstring{$S_n$}{}}

\begin{thm}
    The formula for Eisenstein number of coverings $\pi: X \rightarrow \mathbb{P}^{1}$ with prescribed ramification indices is as follows:

\begin{equation*}
\sum_{\pi: X \rightarrow \mathbb{P}^{1}} \frac{1}{|\operatorname{Aut} \pi|}=\frac{\left|C_{1}\right|\left|C_{2}\right| \ldots\left|C_{k}\right|}{(n !)^{2}} \sum_{\chi} \frac{\chi\left(c_{1}\right) \chi\left(c_{2}\right) \ldots \chi\left(c_{k}\right)}{\left(\chi(1)^{k-2}\right)} 
\end{equation*}
\end{thm}

\begin{proof}
In view of \eqref{Burnside} it is sufficient to show that

\begin{equation*}
\#\left\{c_{1} c_{2} \ldots c_{k}=1 \mid c_{i} \in C_{i} \subset S_{n}\right\}=\sum_{\pi: X \rightarrow \mathbb{P}^{1}} \frac{n !}{\mid \text { Aut } \pi \mid} \
\end{equation*}

    A solution $\left\{c_{1}, c_{2}, \ldots, c_{k}\right\}$ of the equation above corresponds to a ramified covering $\pi: X \rightarrow \mathbb{P}^{1}$ and

$$
\text { Aut } \pi \cong C\left(c_{1}, c_{2}, \ldots, c_{k}\right) \text {, }
$$

where $C\left(c_{1}, c_{2}, \ldots c_{k}\right)$ is the centralizer of the set $\left\{c_{1}, c_{2}, \ldots, c_{k}\right\}$ in $S_{n}$. Hence the number of solutions conjugate to $\left\{c_{1}, c_{2}, \ldots, c_{k}\right\}$ is equal to

$$
\left[S_{n}: C\left(c_{1}, c_{2}, \ldots, c_{k}\right)\right]=\frac{n !}{\mid \text { Aut } \pi \mid}
$$

and the result follows.
\end{proof}

Now if we turn to Belyi coverings $\beta$ with 3 respective conjugacy classes of monodromy permutations for ramification points 0,1 and $\infty$, the formula above will be as follows:

$$
\sum_{\beta: S \rightarrow \mathbb{P}^{1}} \frac{1}{|\operatorname{Aut} \beta|}=\frac{\left|C_{0}\right|\left|C_{1}\right|\left|C_{\infty}\right|}{(n !)^{2}} \sum_{\chi} \frac{\chi\left(c_{1}\right) \chi\left(c_{2}\right) \chi\left(c_{3}\right)}{\chi(1)}
$$

More specifically, this formula for polynomial coverings of degree $n$ with ramification scheme $n, d^{\bullet}, d^{\circ}$ will be:

$$
\frac{\left|C_{n}\right|\left|C_{d^{\bullet}}\right|\left|C_{d^{\circ}}\right|}{(n !)^{2}} \sum_{\chi} \frac{\chi(n) \chi\left(d^{\bullet}\right) \chi\left(d^{\circ}\right)}{\chi(1)}
$$
or equivalently,

\begin{equation*}
\frac{1}{z\left(d^{\bullet}\right) z\left(d^{\circ}\right)} \sum_{0 \leq l \leq n}(-1)^{l} l !(n-l-1) ! \chi_{l}\left(d^{\bullet}\right) \chi_{l}\left(d^{\circ}\right)\,,
\end{equation*}

where $z\left(d^{\bullet}\right)$ is the order of centralizer of a permutation with cycle structure $d^{\bullet}$.

\section{Exceptional Belyi coverings}

\subsection{Rational exceptional Belyi coverings.}

\begin{dfn}
    Let $S_{g}$ be a compact Riemann surface with genus $g$ and

$$
\beta: S_{g} \rightarrow \mathbb{P}^{1}
$$

be a Belyi covering. $\beta$ is said to be an \textit{exceptional Belyi covering} if it is uniquely determined by its ramification scheme. When $g=0$ we call it \textit{rational exceptional Belyi covering}.
\end{dfn}

\begin{rem}
    There exists unique dessin with given degrees of vertices which corresponds to a rational exceptional Belyi covering.
\end{rem}

There are some examples illustrating these exceptional coverings: Klein studied the first three examples by classifying $G \subset \mathbb{P} G L(2, \mathbb{C})$ as cyclic, dihedral, cubic etc. where $G=\operatorname{Aut}\left(\mathbb{P}^{1}, \beta\right)$ in the natural projection $\mathbb{P}^{1} \rightarrow \mathbb{P}^{1} / G$.

\begin{example}(Cyclic Covering)(Infinite Series) 

This is a rational covering

$$
\begin{aligned}
\beta: \mathbb{P}^{1} & \rightarrow \mathbb{P}^{1} \\
z & \mapsto z^{n}
\end{aligned}
$$

The covering is called as ``cyclic" due to the fact that the group $G=\operatorname{Aut}\left(\mathbb{P}^{1}, \beta\right)$ in the natural projection $\mathbb{P}^{1} \rightarrow \mathbb{P}^{1} / G$ is cyclic.

The ramification scheme is of the form $[n][n] \underbrace{[1,1, \ldots, 1]}_{n \text { times }}$.

\end{example}

\begin{example} (Dihedral Covering)(Infinite Series)

$$
\beta: z \mapsto z^{n}+\frac{1}{z^{n}}
$$

The covering is called as ``dihedral" due to the fact that the group $G=\operatorname{Aut}\left(\mathbb{P}^{1}, \beta\right)$ in the natural projection $\mathbb{P}^{1} \rightarrow \mathbb{P}^{1} / G$ is dihedral.

The ramification scheme is $[p, p] \underbrace{[2,2, \ldots, 2]}_{n / 2 \text { times }} \underbrace{[2,2, \ldots, 2]}_{n / 2 \text { times }}$, where $2 p=n$.

\end{example}

The dessins for the cyclic and dihedral coverings are below.
\begin{figure}[htp]
\includegraphics[width=70mm]{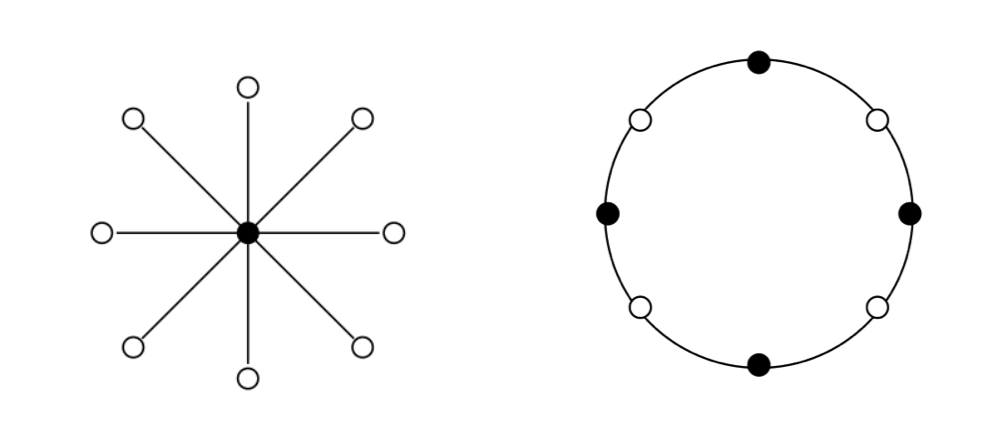}
\end{figure}

\begin{example}
    In 3D space, a \textit{Platonic solid} is a regular, convex polyhedron. The Belyi functions for platonic solids were computed by Felix Klein as stated in \cite{Klein}.

\begin{figure}[htp]
\includegraphics[width=130mm]{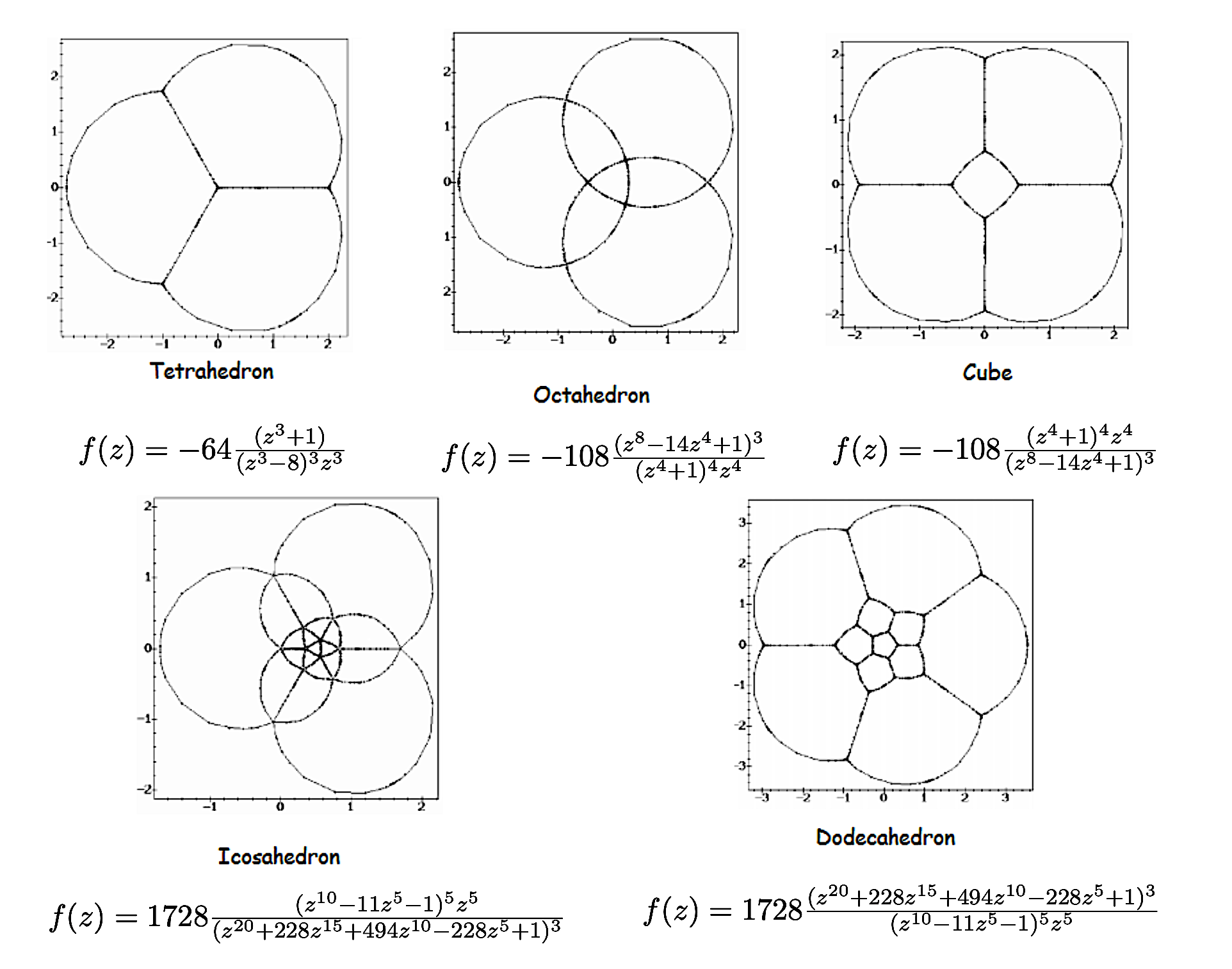}
\end{figure}

\end{example}
\vskip1pt
\begin{example}
     (Chebyshev Covering)(Infinite Series)

The $n$-th Chebyshev polynomial can be expressed as $T_{n}(\cos x)=\cos n x$. The ramification scheme is

$$
\begin{cases}{[n][1,2,2 \ldots, 2][1,2,2 \ldots, 2]} & ; n \text { is odd. } \\ {[n][1,1,2,2 \ldots, 2][2,2, \ldots, 2]} & ; n \text { is even. }\end{cases}
$$

\begin{figure}[htp]
\includegraphics[width=80mm]{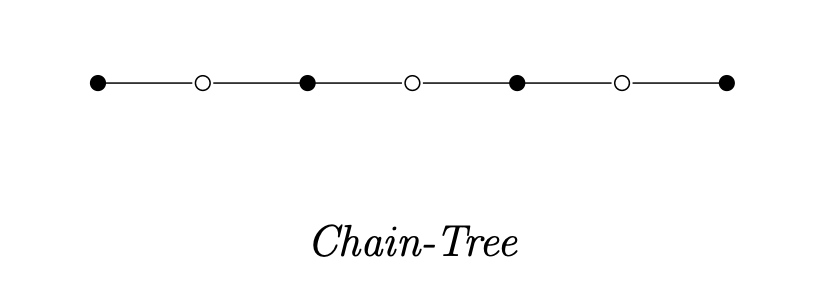}
\end{figure}

\end{example}

\begin{example}We found a new ``interpolating" series between Chebyshev and dihedral covering.

The ramification scheme for this series is

$$
\begin{cases}{[p, q][2,2 \ldots, 2,3][1,2,2, \ldots, 2]} & ; \mathrm{n}=\mathrm{p}+\mathrm{q} \text { is odd. } \\ {[p, q][1,2,2 \ldots, 2,3][2,2 \ldots, 2]} & ; \mathrm{n}=\mathrm{p}+\mathrm{q} \text { is even. }\end{cases}
$$
The respective Belyi function can be expressed as

$$
f(t)=e^{q t i}\left[\frac{\alpha-e^{t i}}{1-\alpha e^{t i}}\right]^{p}+e^{-q t i}\left[\frac{\alpha-e^{-t i}}{1-\alpha e^{-t i}}\right]^{d},
$$

where $\alpha=\frac{q+p}{q-p}$.\\

If $p=q$, this turns out to be a dihedral covering and if $q=0$, then it will be a Chebyshev covering.

\begin{figure}[htp]
\includegraphics[width=60mm]{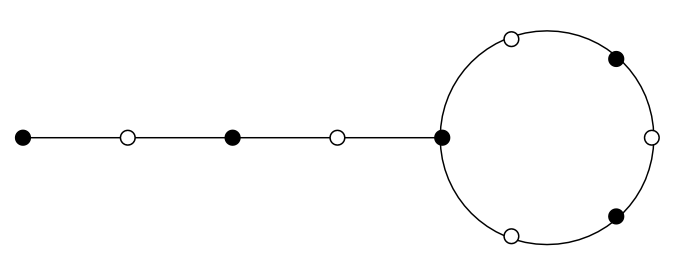}
\end{figure}
\end{example}

We used Maple to calculate the respective Belyi function along with the real graph and dessin when the degree is 9. The Belyi function for corresponding to the ramification scheme $[3,6][2,2,2,3][1,2,2,2,2]$ is
$$
f(t)=e^{6 t i}\left[\frac{3-e^{t i}}{1-3 e^{t i}}\right]^{3}+e^{-6 t i}\left[\frac{3-e^{-t i}}{1-3 e^{-t i}}\right]^{3}
$$
This function has a pole at $\frac{5}{3}$ $\Bigg($In general, the pole is at $\frac{n^{2}+d^{2}}{n^{2}-d^{2}}\Bigg)\,.$

\begin{figure}[htp]
\includegraphics[width=55mm]{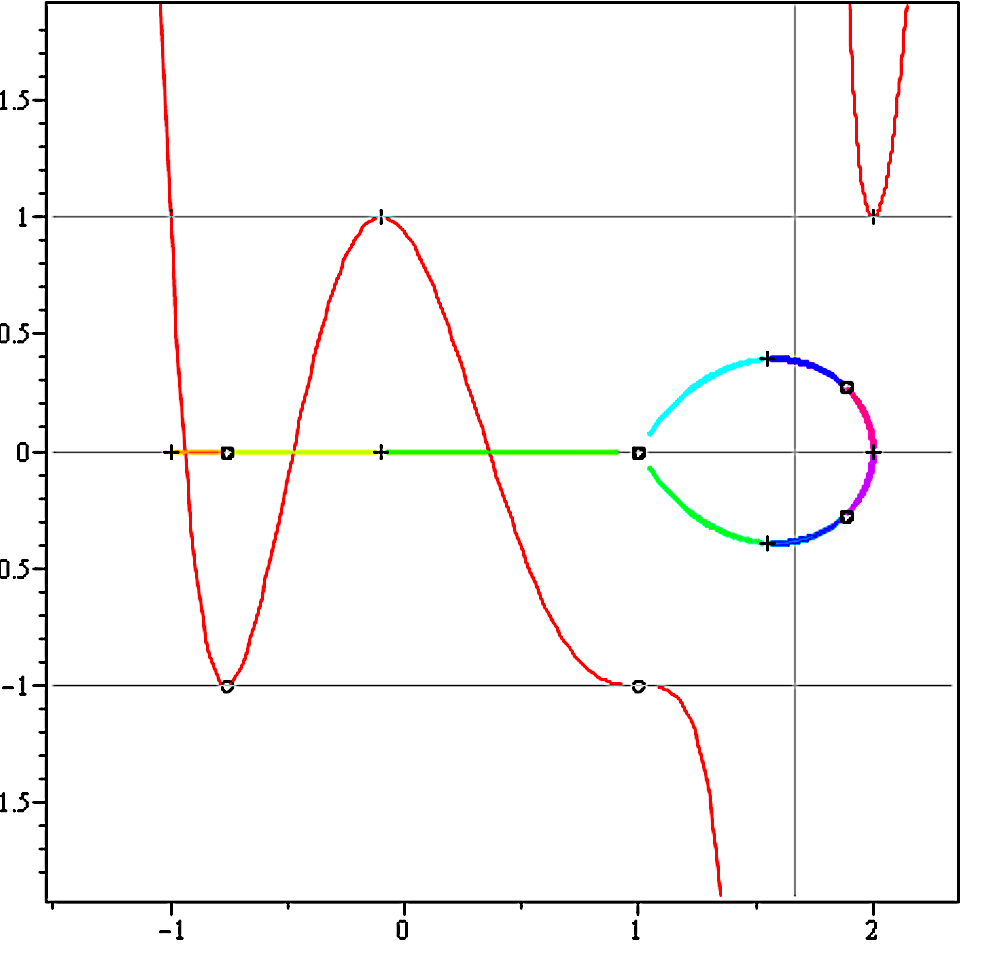}
\end{figure}  
\newpage

\begin{example}
    There is also another new series we found which is of odd degree.

The ramification scheme is $[1, p, p][\underbrace{2,2,2, \ldots, 2}_{p-1 \text { times }}, 3][2,2, \ldots, 2,3]$.

\begin{figure}[htp]
\includegraphics[width=60mm]{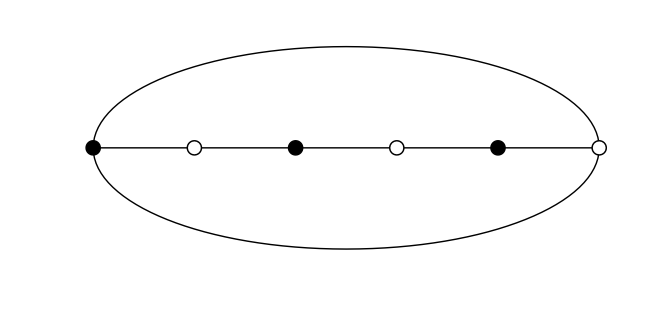}
\end{figure}  

Using Maple for degree 13, we have the very complicated Belyi function for the ramification scheme
$[1,6,6][2,2,2,2,2,3][2,2,2,2,2,3]$ 
\begin{figure}[htp]
\includegraphics[width=150mm]{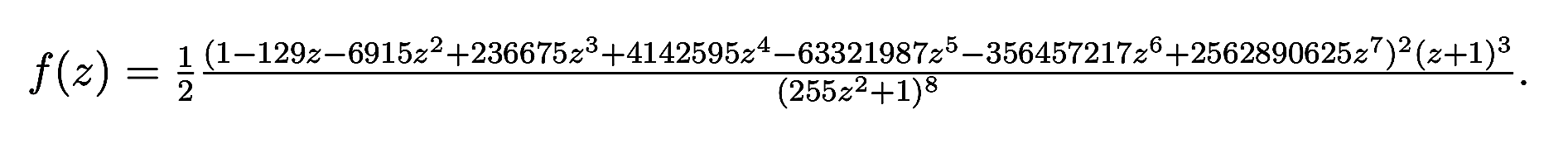}
\end{figure}  

\begin{figure}[htp]
\includegraphics[width=60mm]{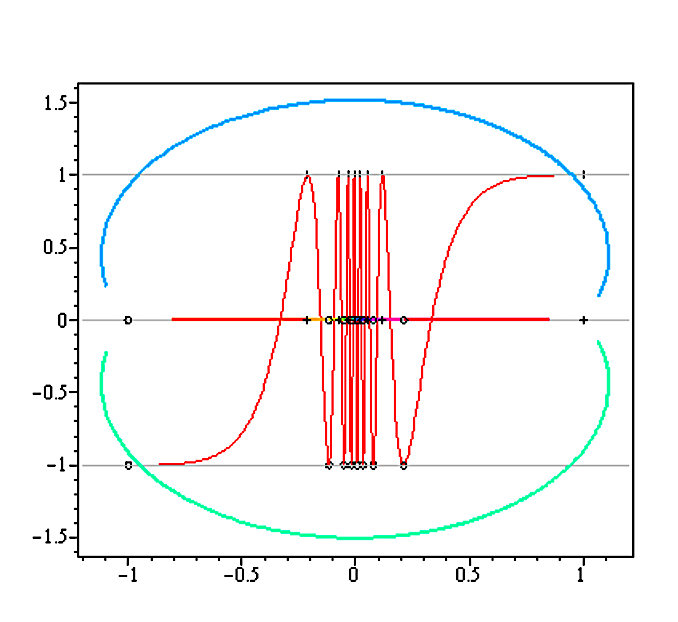}
\end{figure}  

\end{example}
\newpage
\begin{example}
    There are more exceptional series for which we do not know the respective Belyi covering. For example, the covering with the ramification scheme 
    $$[2, p, p][3,3,2,, 2 \ldots, 2][2,2, \ldots, 2]\,.$$ is one of them.
\begin{figure}[htp]
\includegraphics[width=110mm]{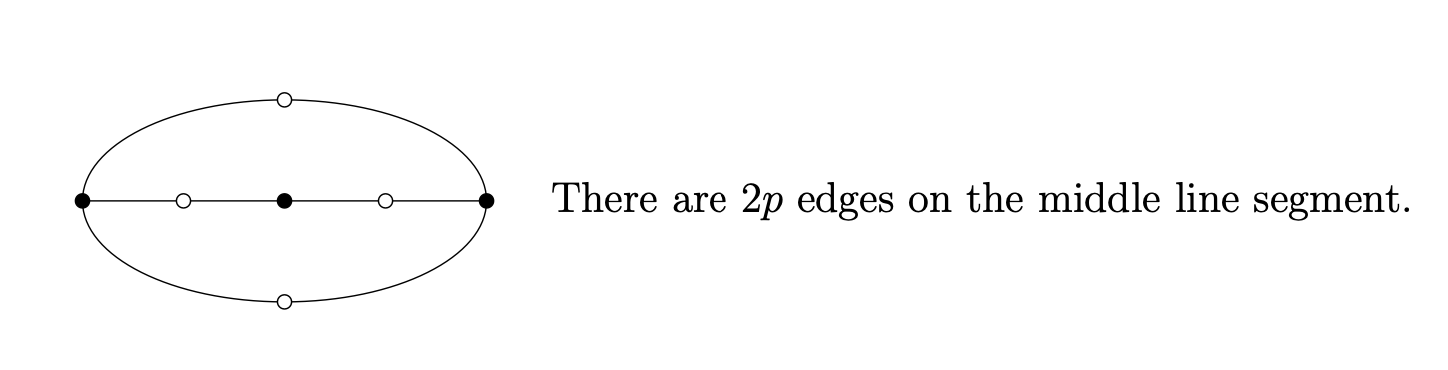}
\end{figure}    
\end{example}
\subsection{Classification of rational exceptional polynomial coverings}

Dessin d'enfants corresponding to polynomial coverings are trees. \\

(1)
\begin{figure}[!htb]
\includegraphics[width=150mm]{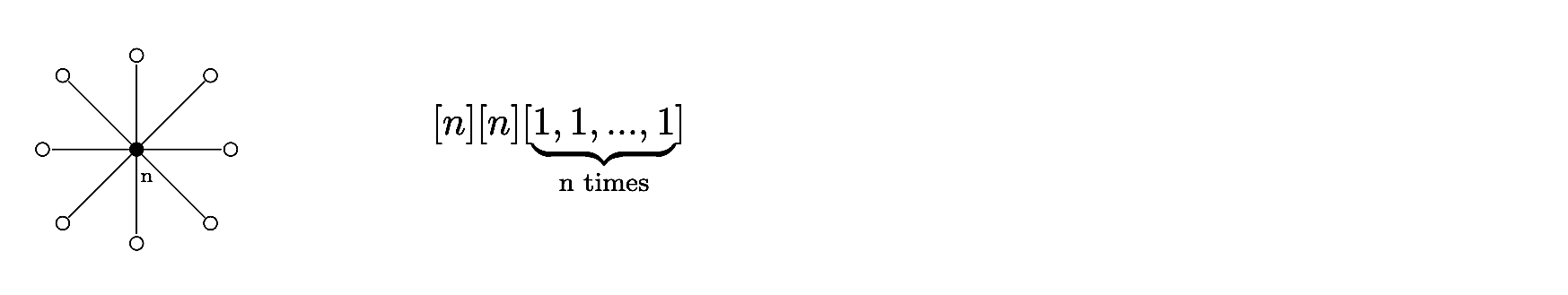}
\end{figure}

(2)\\
\begin{figure}[!htb]
\includegraphics[width=140mm]{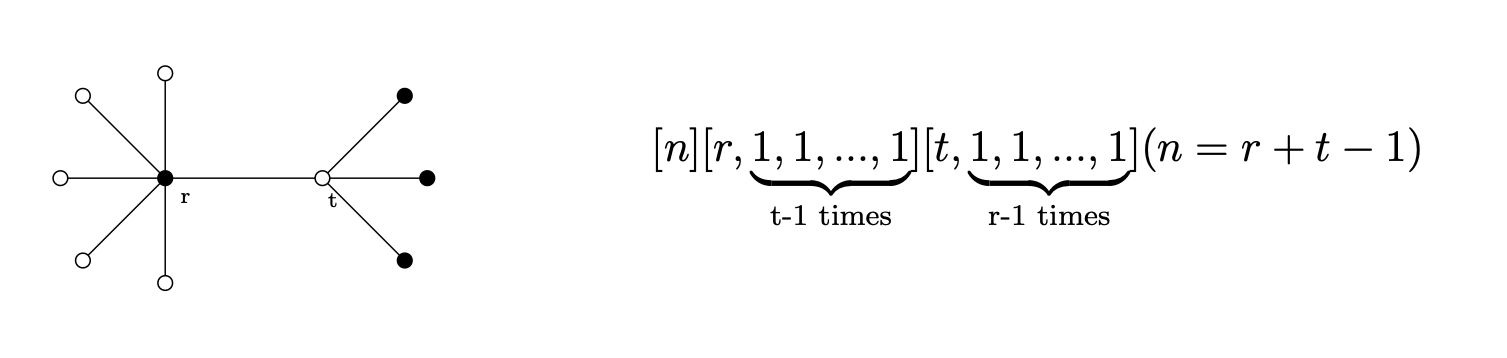}
\end{figure} 

(3)\\
\begin{figure}[!htb]
\includegraphics[width=140mm]{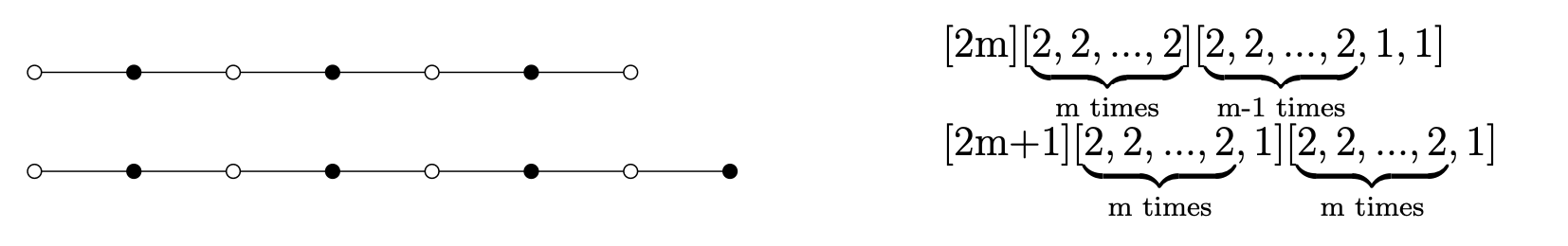}
\end{figure}

\newpage
(4)
\begin{figure}[!htb]
\includegraphics[width=140mm]{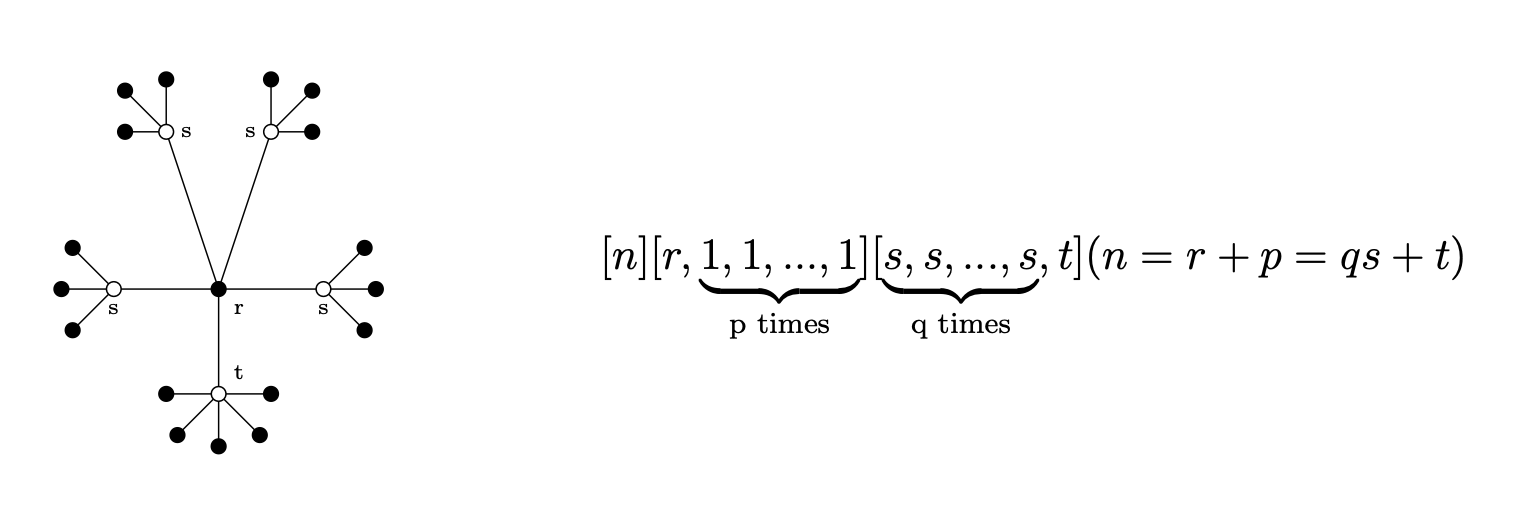}
\end{figure}

(5)
\begin{figure}[!htb]
\includegraphics[width=140mm]{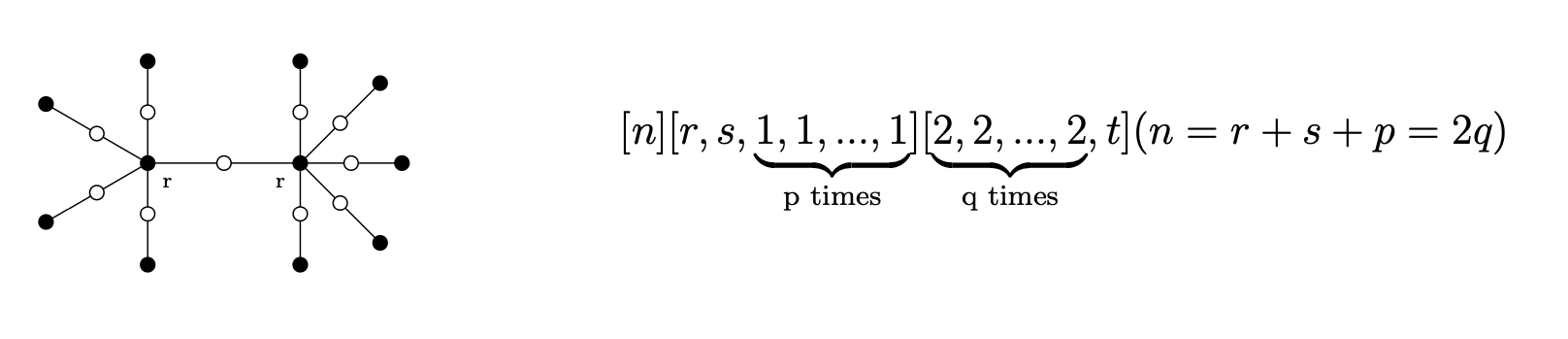}
\end{figure}

(6)
\begin{figure}[!htb]
\includegraphics[width=140mm]{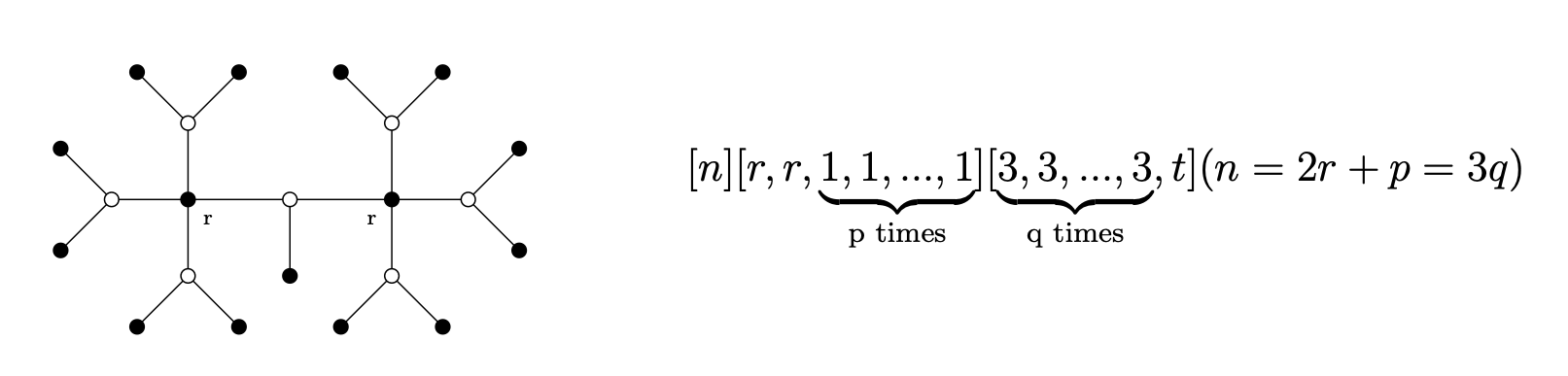}
\end{figure}

(7)
\begin{figure}[!htb]
\includegraphics[width=140mm]{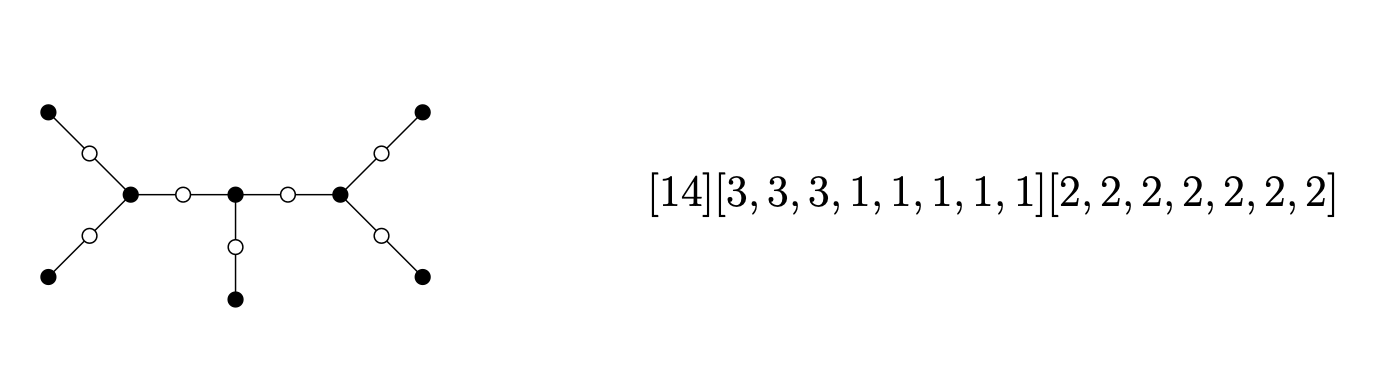}
\end{figure}

\newpage
\subsection{Counting rational exceptional Belyi coverings}

Recall the formulas from Chapter 3:

\begin{equation*}
\sum_{\beta: S \rightarrow \mathbb{P}^{1}} \frac{1}{|\operatorname{Aut} \beta|}=\frac{\left|C_{0}\right|\left|C_{1}\right|\left|C_{\infty}\right|}{(n !)^{2}} \sum_{\chi: \text { irreducible }} \frac{\chi\left(c_{1}\right) \chi\left(c_{2}\right) \chi\left(c_{3}\right)}{\chi(1)} \tag{5.1}
\end{equation*}

where $\chi\left(c_{i}\right)$ are the irreducible characters of permutations $c_{i}$ in the ramification scheme and $\left|C_{i}\right|$ are the size of conjugacy classes with representatives $c_{i}$. Also, we stated Tutte formula for bicolored trees(corresponding to polynomial coverings):

$$
\sum_{T} \frac{1}{|\operatorname{Aut} T|}=\frac{1}{\sigma^{\bullet} \sigma^{\circ}}\left(\begin{array}{l}
\sigma^{\bullet}  \\
q^{\bullet}
\end{array}\right)\left(\begin{array}{l}
\sigma^{\circ} \\
q^{\circ}
\end{array}\right)
$$

where $q_{i}^{\bullet}, q_{i}^{\circ}$ is the number of black and white vertices of degree $i$ respectively and $\sigma^{\bullet}=\sum_{i} q_{i}^{\bullet}, \sigma^{\circ}=\sum_{j} q_{j}^{\circ}$. \\

The sums on the left-hand side in two formulas above simply reduces to the inverse of an integer when the case is exceptional Belyi coverings: By definition, these Belyi coverings are unique, so the Eisenstein number will be $\frac{1}{|\operatorname{Aut} \beta|}$ and $\frac{1}{|\operatorname{Aut} T|}$.
\subsection{Fields of definition of exceptional Belyi coverings
}

\begin{dfn}
 Let $S$ be a compact Riemann surface and $\beta: S \rightarrow \mathbb{P}^{1}$ be a Belyi covering. A \textit{field of definition} of a Belyi pair $(S, \beta)$, or a dessin denfant, is a number field $K$ such that both the algebraic curve $C$ (corresponding to $S$ ) and the Belyi function $\beta$ can be defined with coefficients in $K$.
\end{dfn}

\begin{rem}
    A dessin can have many fields of definition: If some $K$ is a field of definition, every field containing it is also a field of definition.
\end{rem}

\begin{thm}
    The field of definition of an exceptional rational Belyi coverings is either $\mathbb{Q}$ or a quadratic extension of $\mathbb{Q}\,.$
\end{thm}

\vskip5pt

\subsection{Maple calculations for rational exceptional Belyi coverings.}
We developed a Maple algorithm finding all exceptional Belyi coverings with a given genus and degree. This allows us to classify them up to degree 15. This code could be found on the next page. We were able to calculate the respective Belyi functions up to degree 6 completely in addition to some of degree 7. The ramification schemes, Belyi functions and dessins d'enfants can be seen in the table attached. More sample Maple codes based on ``fundamental identities" to find Belyi functions can be found in \cite{Kurk}.  Some rational coverings come from modular curves; e.g. R6.11 in the table. This is the only one we found whose field of definition is different than $\mathbb{Q};$ it is $\mathbb{Q}(\sqrt{10})\,.$ 



\section*{Supplementary material}

```
with(combinat);
with(group);
with(combstruct);
#######################
n:=N: # degree of a covering with three singular fibers
gen:=g: #genus of the covering
print('Genus'=gen,'Degree'=n);
tot:=time():
ABC:=iterstructs(Partition(n+2-2*gen),size=3):##
st:=time():
abc:=nextstruct(ABC): # the number of points in singular fibers
a:=abc[1]:b:=abc[2]:c:=abc[3]:
print('a'=a,'b'=b,'c'=c,'degree'=n);
#n:=a+b+c:       # degree of the covering to ensure that genus=1,
                 # *provided* the covering consists of one component
sol:=0: ## count combinatorail types of covering print('a'=a,'b'=b,'c'=c,degree=n);
unit:=[seq(1,i=1..n)]: ## class of identical permutation
S||n:=permgroup(n, {[[1,2]],[[seq(i,i=1..n)]]}): #symmetric group S_n
P_a := iterstructs(Partition(n),size=a): #ramification indices at fiber a
P_b := iterstructs(Partition(n),size=b): #ramification indices at fiber b
P_c := iterstructs(Partition(n),size=c): #ramification indices at fiber c
while not finished(P_a) do p_a:=nextstruct(P_a):
red_a:=map(x->if x=1 then NULL else x fi,p_a): ## removing cycles of length one
C_a:=SnConjugates(S||n,red_a): # order of the conjugacy class of p_a
while not finished(P_b) do p_b:=nextstruct(P_b):
red_b:=map(x->if x=1 then NULL else x fi,p_b):
C_b:=SnConjugates(S||n,red_b): # order of the conjugacy class of p_b
while not finished(P_c) do p_c:=nextstruct(P_c):
red_c:=map(x->if x=1 then NULL else x fi,p_c):
C_c:=SnConjugates(S||n,red_c): # order of the conjugacy class of p_c
    Rep := iterstructs(Partition(n)): ## Young diagrams = irreps of S_n
    Cov:=0: ## Counting coverings
    while not finished(Rep) do
        rep:=nextstruct(Rep):
Cov:=Cov+Chi(rep,p_a)*Chi(rep,p_b)*Chi(rep,p_c)/Chi(rep,unit):
    od;
        Cov:=Cov*C_a*C_b*C_c/(n!)^2: ## Eisenstein number of coverings

    g:=1/2*(add(x-1, x in p_a)+add(x-1, x in p_b)+add(x-1, x in p_c))-n+1; ## genus, should be = 1.
    if Cov>0 and type(1/Cov,integer)
    then sol:=sol+1:
       print('Ramification indices '= p_a, p_b, p_c);
        print(genus=g,'Eisenstein number of coverings'=Cov):
    fi:
od:
P_c := iterstructs(Partition(n),size=c): #ramification indices at fiber c
od:
P_b := iterstructs(Partition(n),size=b): #ramification indices at fiber b
od:
if sol>0 then print('The number of combinatorial types'=sol, 'time'=time()-st)
else print('No solution','time'=time()-st) fi:
od:
print('Total time'=time()-tot);
```

| | Rational Exceptional Belyi Coverings ||||
|---|---|---|---|
| Label | Ramification Scheme | Belyi Function | Dessins D'enfant |
| R3.1 | \| 3 \| 3 \| 1, 1, 1 \| | $f(z) = z^3$ | |
| R3.2 | \| 3 \| 1, 2 \| 1, 2 \| | $f(z) = 4z^3 - 3z$ | |
| R4.1 | \| 4 \| 4 \| 1, 1, 1, 1 \| | $f(z) = z^4$ | |
| R4.2 | \| 4 \| 2, 2 \| 1, 1, 2 \| | $f(z) = 8z^4 - 8z^2 + 1$ | |
| R4.3 | \| 4 \| 1, 3 \| 1, 1, 2 \| | $f(z) = -(3z - 1)(z + 1)^3$ | |
| R4.4 | \| 1, 3 \| 1, 3 \| 2, 2 \| | $f(z) = e^{3zi}\left[\frac{2-e^{zi}}{1-2e^{zi}}\right] + e^{-3zi}\left[\frac{2-e^{-zi}}{1-2e^{-zi}}\right]$ | |
| R4.5 | \| 2, 2 \| 2, 2 \| 2, 2 \| | $f(z) = z^4 + \frac{1}{z^4}$ | |
| R5.1 | \| 5 \| 5 \| 1, 1, 1, 1, 1 \| | $f(z) = z^5$ | |
| R5.2 | \| 5 \| 1, 2, 2 \| 1, 2, 2 \| | $f(z) = 16z^5 - 20z^3 + 5z$ | |
| R5.3 | \| 5 \| 2, 3 \| 1, 1, 1, 2 \| | $f(z) = (3z - 1)^2(2z + 1)^3$ | |
| R5.4 | \| 5 \| 1, 4 \| 1, 1, 1, 2 \| | $f(z) = (1 - 12z)(3z + 1)^4$ | |
| R5.5 | \| 5 \| 1, 1, 3 \| 1, 1, 3 \| | $f(z) = -(54z^2 + 9z + 1)(3z - 1)^3$ | |

| | | | |
|---|---|---|---|
| R5.6 | \| 5 \| 1, 2, 2 \| 1, 1, 3 \| | $f(z) = (1 - 4z)(6z^2 + 2z + 1)^2$ | |
| R5.7 | \| 1, 4 \| 2, 3 \| 1, 2, 2 \| | $f(z) = e^{4zi} \left[ \dfrac{\frac{5}{3} - e^{zi}}{1 - \frac{5}{3} e^{zi}} \right] + e^{-4zi} \left[ \dfrac{\frac{5}{3} - e^{-zi}}{1 - \frac{5}{3} e^{-zi}} \right]$ | |
| R5.8 | \| 1, 2, 2 \| 2, 3 \| 2, 3 \| | $f(z) = \dfrac{1}{2} \dfrac{(1-9z)^2 (z+1)^3}{(15z^2+1)^2}$ | |
| R6.1 | \| 6 \| 6 \| 1, 1, 1, 1, 1, 1 \| | $f(z) = z^6$ | |
| R6.2 | \| 6 \| 2, 2, 2 \| 1, 1, 2, 2 \| | $f(z) = 32z^6 - 48z^4 + 18z^2 - 1$ | |
| R6.3 | \| 6 \| 1, 5 \| 1, 1, 1, 1, 2 \| | $f(z) = -(5z - 1)(z + 1)^5$ | |
| R6.4 | \| 6 \| 2, 4 \| 1, 1, 1, 1, 2 \| | $f(z) = (2z - 1)^2 (z + 1)^4$ | |
| R6.5 | \| 6 \| 3, 3 \| 1, 1, 1, 1, 2 \| | $f(z) = (z^2 + 1)^3$ | |
| R6.6 | \| 6 \| 1, 1, 4 \| 1, 1, 1, 3 \| | $f(z) = (40z^2 - 8z + 1)(2z + 1)^4$ | |
| R6.7 | \| 6 \| 2, 2, 2 \| 1, 1, 1, 3 \| | $f(z) = (z^3 + 1)^2$ | |
| R6.8 | \| 3, 3 \| 2, 2, 2 \| 2, 2, 2 \| | $f(z) = \dfrac{(z^3+1)^2}{4z^3}$ | |
| R6.9 | \| 1, 5 \| 1, 2, 3 \| 2, 2, 2 \| | $f(z) = e^{5zi} \left[ \dfrac{\frac{3}{2} - e^{zi}}{1 - \frac{3}{2} e^{zi}} \right] + e^{-5zi} \left[ \dfrac{\frac{3}{2} - e^{-zi}}{1 - \frac{3}{2} e^{-zi}} \right]$ | |

| | | | |
|---|---|---|---|
| R6.10 | \| 2, 4 \| 1, 2, 3 \| 2, 2, 2 \| | $f(z) = e^{4zi}\left[\dfrac{\frac{5}{3}-e^{zi}}{1-\frac{5}{3}e^{zi}}\right]^2 + e^{-4zi}\left[\dfrac{\frac{5}{3}-e^{-zi}}{1-\frac{5}{3}e^{-zi}}\right]^2$ | |
| R6.11 | \| 3, 3 \| 1, 2, 3 \| 1, 1, 4 \| | $f(z) = -\dfrac{729\left(z-\frac{\alpha}{2}+\frac{9}{2}\right)(z+\alpha)^2\left(z+\frac{\alpha}{6}-\frac{35}{18}\right)^3}{6400(3\alpha-8)z^3},\ \alpha = \frac{13}{3}+\frac{4\sqrt{10}}{3}$ | |
| R6.12 | \| 3, 3 \| 2, 2, 2 \| 1, 1, 4 \| | $f(z) = -\dfrac{(z^3+9z^2-9z-1)^2}{256z^3}$ | |
| R6.13 | \| 1, 5 \| 3, 3 \| 1, 1, 2, 2 \| | $f(z) = \dfrac{(z^2+10z+5)^3}{1728z}$ | |
| R6.14 | \| 1, 5 \| 2, 2, 2 \| 1, 1, 4 \| | $f(z) = \dfrac{(z^3-5z^2+15z+5)^2}{(256z)}$ | |
| R6.15 | \| 2, 4 \| 3, 3 \| 1, 1, 1, 3 \| | $f(z) = -\dfrac{(z^2+8z-2)^3}{729z^2}$ | |
| R6.16 | \| 1, 5 \| 3, 3 \| 1, 1, 1, 3 \| | f(z)$=-\dfrac{(z^2-5z-5)^3}{729z}$ MODULAR | |
| R6.17 | \| 2, 4 \| 3, 3 \| 1, 1, 2, 2 \| | $f(z) = \dfrac{(z^2-2)^3}{27z^2}$ MODULAR | |
| R6.18 | \| 3, 3 \| 3, 3 \| 1, 1, 2, 2 \| | $f(z) = -\dfrac{(z^2+2z-3)^3}{64z^3}$ MODULAR | |
| R7.1 | \| 7 \| 2, 2, 3 \| 1, 1, 1, 1, 3 \| | $f(z) = -(30z^2+6z+1)^2(4z-1)^3$ | |
| R7.2 | \| 7 \| 1, 1, 1, 1, 3 \| 1, 3, 3 \| | $f(z) = (1-3z)(2z^2+z+1)^3$ | |
| R7.3 | \| 7 \| 1, 2, 2, 2 \| 1, 1, 1, 4 \| | $f(z) = (1-4z)(20z^3+6z^2+2z+1)^2$ | |



\subsection{Acknowledgements}
I thank my supervisor Professor Alexander Klyachko for his valuable support and guidance.
\bibliographystyle{alpha}
\bibliography{mybib}

\begin{thebibliography}{G{\'a}D11}

\bibitem[Adr09]{A09a}
N.M. Adrianov.
\newblock On plane trees with a prescribed number of valency set realizations.
\newblock {\em J. Math. Sci.}, 158(1):5--10, 2009.

\bibitem[Bel80]{Belyi}
G.V. Bely{\u i}.
\newblock On galois extensions of a maximal cyclotomic field.
\newblock {\em Math. USSR Izv.}, 14(2):247--256, 1980.

\bibitem[BZ96]{BetZvo}
J.~B{\'e}tr{\'e}rna and A.~Zvonkin.
\newblock Plane trees and shabat polynomials.
\newblock {\em Disc. Math.}, 153(1-3):47--58, 1996.

\bibitem[CG94]{CouGran}
J.M. Couveignes and L.~Granboulan.
\newblock Dessins from a geometric point of view, the grothendieck theory of dessins d'enfants.
\newblock {\em London Math. Soc. Lecture Note Ser.}, 200:79--113, 1994.

\bibitem[F.K07]{Klein}
F.Klein.
\newblock Lectures on the icosahedron and the solution of the fifth degree.
\newblock 2007.

\bibitem[G{\'a}D11]{GirGon}
E.~Girondo and G.~Gonz {\'a}lez Diez.
\newblock ntroduction to compact riemann surfaces and dessins d’enfants.
\newblock {\em London Math. Soc. Student Texts}, 79, 2011.

\bibitem[Gro03]{Groth}
A.~Grothendieck.
\newblock Sketch of a programme.
\newblock 2003.

\bibitem[KK98]{KlyKur}
A.~Klyachko and E.~Kurtaran.
\newblock Some identities and asymptotics for characters of the symmetric group.
\newblock {\em J. Algebra}, 206(2):413--437, 1998.

\bibitem[K{\"o}c04]{Kock}
B.~K{\"o}ck.
\newblock Belyi's theorem revisited.
\newblock 45(1):253--265, 2004.

\bibitem[K{\"u}r15]{Kurk}
C.~K{\"u}rkoglu.
\newblock Exceptional belyi coverings.
\newblock {\em ProQuest Publishing}, 2015.

\bibitem[LS04]{LieSha}
M.~W. Liebeck and A.~Shalev.
\newblock Fuchsian groups, coverings of riemann surfaces, subgroup growth, random quotients and random walks.
\newblock {\em J. Algebra}, 276(2):552--601, 2004.

\bibitem[LZ04]{LanZvo}
S.~Lando and A.~Zvonkin.
\newblock Graphs on surfaces and their applications.
\newblock 141:xv+455, 2004.

\bibitem[MZ00]{MagZvo}
N.~Magot and A.~Zvonkin.
\newblock Belyi functions for archimedean solids.
\newblock {\em Disc. Math.}, 217:249--271, 2000.

\bibitem[Ser73]{Serre2}
J.P. Serre.
\newblock A course in arithmetic.
\newblock 7:ix+119, 1973.

\bibitem[Ser97]{Serre3}
J.P. Serre.
\newblock Galois cohomology.
\newblock pages x+212, 1997.

\bibitem[Sha94]{Sha}
G.~Shabat.
\newblock On the classification of plane trees by their galois orbits in the grothendieck theory of dessins d’enfants.
\newblock {\em London Math. Soc. Lecture Note Ser.}, 200:169--177, 1994.

\bibitem[SS03]{SinSyd}
D.~Singerman and R.I. Syddall.
\newblock The riemann surface of a uniform dessin.
\newblock 44(2):413--430, 2003.

\bibitem[SV90]{ShaVoe}
G.V. Shabat and V.A. Voevodsky.
\newblock Drawing curves over number fields.
\newblock 88:199--227, 1990.

\bibitem[Tut64]{Tut}
W.~T. Tutte.
\newblock The number of planted plane trees with a given partition.
\newblock 71(3):272--277, 1964.

\bibitem[Zvo]{Zvo}
A.~Zvonkin.
\newblock Belyi functions:examples, properties, and applications.

\end{thebibliography}

\end{document}